\documentclass[oneside,11pt]{amsart}
\usepackage[utf8]{inputenc}
\usepackage[T1]{fontenc}
\usepackage{amsmath, mathrsfs, bbm}

\usepackage{url}

\usepackage{amssymb}
\usepackage{amsxtra}

\newtheorem{thm}{Theorem}[section]

\newtheorem{prop}[thm]{Proposition}
\newtheorem{lem}[thm]{Lemma}
\newtheorem{cor}[thm]{Corollary}
\newtheorem{conj}[thm]{Conjecture}
\newtheorem*{thmsurface}{Theorem~\ref{thm:VirtualSurfaceGroup}}
\newtheorem*{thmextend}{Theorem~\ref{thm:RigidExtends}}
\newtheorem*{corpair}{Corollary~\ref{cor:NoParabolicExact}} 
\theoremstyle{definition}
\newtheorem{defn}[thm]{Definition}
\newtheorem{rem}[thm]{Remark}

\renewcommand{\bar}[1]{\overline{#1}}
\newcommand{\boundary}{\partial}

\renewcommand{\emptyset}{\varnothing}

\newcommand{\field}[1]{\mathbb{#1}}
\newcommand{\Z}{\field{Z}}
\newcommand{\R}{\field{R}}

\newcommand{\E}{\field{E}}

\newcommand{\Hyp}{\field{H}}

\renewcommand{\hat}{\widehat}

\DeclareMathOperator{\CAT}{CAT}

\DeclareMathOperator{\Stab}{Stab}

\DeclareMathOperator{\diam}{diam}

\usepackage{combelow}
\newcommand{\Drutu}{Dru{\cb{t}}u}

\newcommand{\Haissinsky}{Ha\"{i}ssin\-sky}
\newcommand{\Sierpinski}{Sier\-pi{\'n}\-ski}

\usepackage{color}
\usepackage{pinlabel}
\usepackage{todonotes}

\usepackage{ifthen}

\newcommand{\showcomments}{yes}

\newsavebox{\commentbox}
{\ifthenelse{\equal{\showcomments}{yes}}%
{\footnotemark
    \begin{lrbox}{\commentbox}
    \begin{minipage}[t]{1.25in}\raggedright\sffamily\upshape\tiny
    \footnotemark[\arabic{footnote}]}
{\begin{lrbox}{\commentbox}}}%
{\ifthenelse{\equal{\showcomments}{yes}}%
{\end{minipage}\end{lrbox}\marginpar{\usebox{\commentbox}}}
{\end{lrbox}}}

\begin{document}

\title[Planar boundaries and parabolic subgroups]{Planar boundaries and parabolic subgroups}

\author[G.C.~Hruska]{G.~Christopher Hruska}
\address{Department of Mathematical Sciences\\
University of Wisconsin--Milwaukee\\
PO Box 413\\
Milwaukee, WI 53211\\
USA}
\email{chruska@uwm.edu}

\author[G.S.~Walsh]{Genevieve S.~Walsh }
\address{Department of Mathematics\\
Tufts University\\
Medford, MA 02155\\
USA}
\email{genevieve.walsh@tufts.edu}

\begin{abstract}
We study the Bowditch boundaries of relatively hyperbolic group pairs, focusing on the case where there are no cut points.
We show that if $(G,\mathcal{P})$ is a rigid relatively hyperbolic group pair whose boundary embeds in $S^2$, then the action on the boundary extends to a convergence group action on $S^2$.
More generally, if the boundary is connected and planar with no cut points, 
we show that every element of $\mathcal{P}$ is virtually a surface group.
This conclusion is consistent with the conjecture that such a group $G$ is virtually Kleinian. We give numerous examples to show the necessity of our assumptions. 
\end{abstract}



\date{\today}

\maketitle

\setcounter{tocdepth}{2}
\tableofcontents

\section{Introduction}
\label{sec:Introduction}

Relatively hyperbolic groups generalize the notion of geometrically finite Kleinian groups acting on real hyperbolic space $\Hyp^3$ to groups acting similarly on other $\delta$--hyperbolic spaces \cite{BowditchRelHyp}.
Convergence groups were introduced by Gehring--Martin \cite{GehringMartin87} for actions on $S^2 = \boundary \Hyp^3$
and were related to boundaries of $\delta$--hyperbolic spaces by Tukia and Freden \cite{Tukia94,Freden95}.
In this article, we study geometrically finite convergence groups, 
which are precisely the boundary actions associated to relatively hyperbolic group pairs by \cite{Yaman04,GerasimovPotyagailo15_NonFG}.
We are motivated by the following general question: What conditions on a relatively hyperbolic pair $(G,\mathcal{P})$ with planar Bowditch boundary are sufficient to ensure that $G$ is a Kleinian group?
Notice that if a group has a planar boundary, the action of the group on its boundary might not extend to an action on $S^2$.

It is easy to construct examples of relatively hyperbolic groups with planar boundary that are not virtually fundamental groups of $3$--manifolds. 
We discuss one such construction, in Proposition~\ref{prop:HidingStuff}, which is general enough that the peripheral subgroups can be any arbitrary non-torsion group.
A second, cautionary example in Proposition~\ref{prop:ThreeSlopes} has peripheral subgroups equal to $\mathbb{Z}\oplus\mathbb{Z}$, yet is still not a virtual $3$--manifold group.
These two constructions produce groups whose Bowditch boundaries have cut points.  For these groups, the action on the Bowditch boundary does not extend to an action on $S^2$.

However, we prove that, for rigid group pairs,  the action on the Bowditch boundary does extend to a geometrically finite convergence group action on $S^2$. 
A relatively hyperbolic group pair $(G, \mathcal{P})$ is {\it rigid} if $G$ has no elementary splittings relative to $\mathcal{P}$. See Definition~\ref{def:Rigid} for more explanation. 

\begin{thmextend}
Suppose $G$ is one ended and $(G, \mathcal{P})$ is relatively hyperbolic. If $(G, \mathcal{P})$ is rigid 
and $M=\boundary(G,\mathcal{P})$ topologically embeds in $S^2$, then the action of $G$ on $M$ extends to an action on $S^2$ by homeomorphisms.
\end{thmextend}

In studying the Kleinian question mentioned above, it would be useful to determine which subgroups may arise as peripheral subgroups of a relatively hyperbolic group with planar boundary having no cut points.
Note that groups with no cut points in the boundary are not necessarily rigid, nor does the action always extend to $S^2$.   This is shown explicitly in \cite{KapovichKleiner00,
HruskaStarkTran_DontAct}.

Using Theorem~\ref{thm:RigidExtends}, we characterize the peripheral subgroups in the one-ended case, even though the action may not extend to $S^2$. This result is consistent with the conjecture that such groups are virtually Kleinian. A \emph{surface group} is either the fundamental group of a closed surface or a finitely generated free group.
 
\begin{thmsurface}
Suppose $G$ is one ended, and suppose $(G, \mathcal{P})$ is relatively hyperbolic such that the boundary $\partial(G, \mathcal{P})$ is planar and without cut points.  Then each $P \in \mathcal{P}$ is virtually a surface group. 
\end{thmsurface}

We note that if $(G,\mathcal{P})$ is relatively hyperbolic and $\boundary(G,\mathcal{P})$ is connected with no cut points, then $G$ must be finitely generated (see Theorem~\ref{thm:Connected}).
There are relatively hyperbolic group pairs satisfying the hypotheses such that the peripheral groups are higher genus surface groups. 

Theorem~\ref{thm:VirtualSurfaceGroup} would be more straightforward if one knew that the action of $G$ on the Bowditch boundary extends to an action on $S^2$, which is clear in several special cases previously studied: when the boundary is the $2$--sphere itself or the \Sierpinski\ carpet \cite[Thm.~0.3]{Dahmaniparabolic} or a Schottky set \cite{HPWpreprint}.  
The main difficulty in the proof of Theorem~\ref{thm:VirtualSurfaceGroup} is reducing to a case in which the action of a suitable subgroup of $G$ extends to $S^2$.

Theorem~\ref{thm:VirtualSurfaceGroup} 
provides evidence for the following conjecture, which extends the Cannon Conjecture \cite{Cannon91}.

\begin{conj}
\label{conj:VirtuallyKleinian}
Suppose $G$ is one ended and $(G,\mathcal{P})$ is relatively hyperbolic. If $\partial(G, \mathcal{P})$ is planar with no cut points,
then $G$ is virtually a Kleinian group. 
\end{conj}

The case when the boundary is a $2$--sphere is the Relative Cannon Conjecture \cite{GrovesManningSisto,TshWa}, and the case when the boundary is a \Sierpinski\ carpet is a conjecture due to Kapovich--Kleiner \cite{KapovichKleiner00}.
The word ``virtually'' may be dropped in the case of the $2$--sphere or the \Sierpinski\ carpet---provided that one interprets ``Kleinian group'' to mean a group acting properly and isometrically on $\Hyp^3$.  However examples of Kapovich--Kleiner \cite{KapovichKleiner00} and Hruska--Stark--Tran \cite{HruskaStarkTran_DontAct} illustrate that groups with planar boundary as in the conjecture need not be Kleinian (even in the hyperbolic setting) so a virtual assumption is necessary. A key special case of Conjecture~\ref{conj:VirtuallyKleinian} is proved by \Haissinsky\ in \cite{Haissinsky_Invent}, the case when $G$ is a hyperbolic group such that $\partial G$ does not contain a \Sierpinski\ carpet and $G$ has no $2$--torsion.

A related conjecture of Martin--Skora \cite[Conj.~6.2]{MartinSkora89} states that any convergence group acting on $S^2$ is covered by a Kleinian group. This conjecture would not directly imply Conjecture~\ref{conj:VirtuallyKleinian} since the action of a relatively hyperbolic group on its planar boundary may not extend to an action on $S^2$.

To prove Theorem~\ref{thm:RigidExtends},
we first need a complete understanding of all cut pairs in the boundary of a relatively hyperbolic group, in order to conclude that rigid groups do not have cut pairs.  

Bowditch and Dasgupta--Hruska's proof that the Bowditch boundary is locally connected involves a general classification of the cut points of the boundary.  A major ingredient in the proof is the theorem that a cut point must always be the fixed point of a parabolic subgroup \cite{Bowditch99Connectedness,DasguptaHruska_LC}.

When considering cut pairs, work of Haulmark--Hruska \cite{HaulmarkHruska_JSJ} shows that the inseparable, loxodromic cut pairs of the boundary are closely related to splittings over $2$--ended groups.  A cut pair $\{x,y\}$ of a Peano continuum is \emph{inseparable} if $x$ and $y$ cannot be separated by any other cut pair.
It is \emph{loxodromic} if $x$ and $y$ are the fixed points of a loxodromic group element.
By analogy with the case of cut points, one might conjecture that inseparable cut pairs must always be loxodromic.
A theorem of Haulmark \cite{HaulmarkRelHyp} reduces this conjecture to showing that the boundary does not contain an inseparable \emph{parabolic} cut pair, \emph{i.e.}, an inseparable cut pair consisting of two parabolic points.

We were surprised to discover (many) relatively hyperbolic groups with planar boundaries that do contain inseparable parabolic cut pairs.   We discuss examples of inseparable parabolic cut pairs in Bowditch boundaries in Section~\ref{sec:Examples}.  
Then we show in Section~\ref{sec:inseplox} that this pathology can occur only if $G$ splits over a finite group.

\begin{corpair}
Suppose $(G,\mathcal{P})$ is relatively hyperbolic and $M=\boundary(G,\mathcal{P})$ is connected with no cut points.
If $G$ is one ended, then all inseparable cut pairs of $M$ are loxodromic. 
\end{corpair}

The corollary does not involve planarity, but rather applies broadly, far beyond the low-dimensional setting of Conjecture~\ref{conj:VirtuallyKleinian}.
This corollary plays a key role in the proof of Theorem~\ref{thm:VirtualSurfaceGroup}, allowing us to deduce that rigid pieces have no cut pairs.
This conclusion allows us to establish that the action on the planar boundary extends to an action on $S^2$.

In Section~\ref{sec:prelim} we discuss background on convergence groups, relatively hyperbolic groups, properties of their boundaries, and some special subgroups. 
In Section~\ref{sec:ProperActions}, we review why proper actions by homeomorphisms on the plane are tame.
Section~\ref{sec:Examples} is dedicated to the examples discussed above.  In Section~\ref{sec:Peano} we constrict a simplicial tree dual to the family of all inseparable cut pairs in the boundary. This tree is a key tool for relating inseparable parabolic cut pairs to splittings over finite groups in Section~\ref{sec:inseplox}.  We further develop this connection in Section~\ref{sec:ParabolicCutPairs}, where we show in Theorem~\ref{thm:nocutpairs}  
that rigid one-ended relatively hyperbolic group boundaries (which are not $S^1$) do not have cut pairs.
Finally, in Section~\ref{sec:PeripheralSubgroups}, we prove Theorems \ref{thm:VirtualSurfaceGroup} and~\ref{thm:RigidExtends}.

\subsection{Acknowledgements}
The authors thank Peter \Haissinsky\ and Craig Guilbault for helpful discussions and also thank the referee for helpful feedback.
The first author was partially supported by grants \#318815 and \#714338 from the Simons Foundation, and the second author by NSF grant DMS \#1709964. 

\section{Preliminaries}
\label{sec:prelim}

This section collects various background results from the literature.

\begin{defn} \label{def:convergence} 
A \emph{convergence group action} of a countable group $G$ on a metrizable compactum $M$ is an action by homeomorphisms such that for any sequence $(g_i)$ of distinct elements in $G$
there is a subsequence $(g_{n_i})$ such that  there exist points $\zeta,\xi \in M$ such that
\[
   g_{n_i} \big| \bigl( M\setminus\{\zeta\}\bigr) \to \xi
\]
uniformly on compact sets. Such a subsequence is a \emph{collapsing subsequence}.

A point $\zeta \in M$ is a \emph{conical limit point} if there exists a sequence $(g_i)$ in $G$ and a pair of distinct points $\xi_0\ne\xi_1 \in M$ such that
\[
   g_i \big| \bigl( M \setminus \{\zeta\}\bigr) \to \xi_0
   \qquad \text{and} \qquad
   g_i(\zeta) \to \xi_1.
\]
A point $\eta\in M$ is \emph{bounded parabolic} if
its stabilizer acts properly and cocompactly on $M \setminus\{\eta\}$.
A convergence group acting on $M$ is \emph{geometrically finite} if every point of $M$ is either a conical limit point or a bounded parabolic point.  The stabilizers of the bounded parabolic points are called \emph{maximal parabolic subgroups}.  
\end{defn}

\begin{defn}[Relatively hyperbolic]
A \emph{group pair} consists of a group $G$ and a family $\mathcal{P}$ of infinite subgroups that is closed under conjugation.
A group pair $(G,\mathcal{P})$ is \emph{relatively hyperbolic} if 
$G$ admits a geometrically finite convergence group action on a metrizable compactum $M$ such that
$\mathcal{P}$ is the set of maximal parabolic subgroups.
If the pair $(G,\mathcal{P})$ is relatively hyperbolic, the family $\mathcal{P}$ is a \emph{peripheral structure} and the subgroups $P \in \mathcal{P}$ are \emph{peripheral subgroups} of $(G,\mathcal{P})$.
\end{defn}

Any two compacta $M$ and $M'$ as above are $G$--equivariantly homeomorphic by \cite{Yaman04,BowditchRelHyp,GerasimovPotyagailo16_Similar}.
The \emph{Bowditch boundary} of a relatively hyperbolic pair  $\partial (G, \mathcal{P})$ is defined to be any metrizable compactum $M$ admitting a geometrically finite action as above.

\begin{defn}
\label{def:RelQC}
Assume $(G,\mathcal{P})$ is relatively hyperbolic.
A subgroup $H\le G$ is \emph{elementary} if the limit set $\Lambda H$ of $H$ in $\boundary (G,\mathcal{P})$ has fewer than three points.
A subgroup $H\le G$ is \emph{relatively quasiconvex} if $H$ is elementary or if the action of $H$ on its limit set is a geometrically finite convergence action.
A relatively quasiconvex subgroup $H$ inherits a natural relatively hyperbolic structure $(H, \mathcal{P}_H)$ where $\mathcal{P}_H$ is the set of all infinite subgroups of the form $H\cap P$ for $P \in \mathcal{P}$.  Furthermore, the Bowditch boundary of $(H, \mathcal{P}_H)$ is the limit set of $H$ in $\partial(G, \mathcal{P})$.  See \cite{Dahmani03Combination,Hruska10RelQC} for more information.  
\end{defn} 

A \emph{Peano continuum} is a connected, locally connected compact metrizable space. A \emph{cut point} of a connected space $X$  is a point $x \in X$ such that $X\setminus x$ is not connected.  A \emph{local cut point} is a point that is a cut point of some connected open subset of $X$.

The following result was established by Bowditch \cite{Bowditch99Connectedness} under hypotheses on the peripheral groups and by Dasgupta--Hruska \cite{Dasgupta_Thesis,DasguptaHruska_LC} in the general case.

\begin{thm}[Splittings and the boundary]
\label{thm:Connected}
Let $(G,\mathcal{P})$ be relatively hyperbolic.
The boundary $M = \boundary(G,\mathcal{P})$ is connected if and only if $G$ is one ended relative to $\mathcal{P}$; \emph{i.e.}, $G$ does not split relative to $\mathcal{P}$ over a finite subgroup \textup{(}see Definition~\ref{def:relsplitting}\,\textup{)}.

Suppose $M$ is connected.
Then $M$ is a Peano continuum.
Furthermore, $M$ has no cut point if and only if $G$ does not split relative to $\mathcal{P}$ over a parabolic subgroup, in which case $G$ is finitely generated.
\end{thm}

\begin{rem}[Changing the peripheral structure]
\label{rem:ChangingPeripheral}
If $(G,\mathcal{P})$ is relatively hyperbolic, one can add finitely many conjugacy classes of maximal, nonparabolic two-ended subgroups to $\mathcal{P}$ to form a new group pair $(G,\mathcal{P}')$ that is again relatively hyperbolic (see Dahmani and Osin \cite{Dahmani03Combination,Osin06_Elementary}).
Conversely, if $(G,\mathcal{P}')$ is relatively hyperbolic, and $\mathcal{P}$ is formed from $\mathcal{P}'$ by removing finitely many conjugacy classes of two-ended subgroups then $(G,\mathcal{P})$ is relatively hyperbolic. For a proof, see \Drutu--Sapir \cite{DrutuSapirTreeGraded} if $G$ is finitely generated and Matsuda--Oguni--Yamagata \cite{Matsudarel} in general.

The Bowditch boundary does change when one changes the peripheral structure, as shown in a general setting by Wen-yuan Yang \cite{Yang14_Peripheral}.  For instance, if one adds finitely many conjugacy classes of maximal two-ended nonparabolic subgroups to $\mathcal{P}$, then the new boundary $\partial(G, \mathcal{P}')$ is obtained from $\partial(G, \mathcal{P})$ by identifying the limit set of each $P \in \mathcal{P}' \setminus \mathcal{P}$ in $\partial(G, \mathcal{P})$ to a point. See Dahmani \cite{Dahmani03Combination} for related results.
\end{rem}

\begin{defn}
\label{def:relsplitting}  
Suppose $(G,\mathcal{P})$ is relatively hyperbolic. A \emph{splitting relative to $\mathcal{P}$} is an action of $G$ on a simplicial tree $T$ without inversions such that each peripheral subgroup $P \in \mathcal{P}$ stabilizes a vertex of $T$.
\end{defn} 

Let $\mathcal{E}$ be a set of subgroups of $G$. An $\mathcal{E}$--splitting relative to $\mathcal{P}$ is a splitting of $G$ relative to $\mathcal{P}$ where each edge stabilizer is in $\mathcal{E}$. 
An \emph{elementary splitting relative to $\mathcal{P}$} is the case where $\mathcal{E}$ is the family of all elementary subgroups.

\begin{defn}[Pinched peripheral structure]
\label{def:PinchedPeripheral}
If $G_v$ is a vertex stabilizer of an elementary splitting relative to $\mathcal{P}$, there is a natural peripheral structure obtained by adding the edge stabilizers of incident edges in the splitting.  Since all edge groups are elementary, and hence relatively quasiconvex, each vertex stabilizer $G_v$ is also relatively quasiconvex (see Bigdely--Wise \cite[Lem.~4.9]{BigdelyWise_Combination} or Guirardel--Levitt \cite[Prop.~3.4]{GuirardelLevitt15_AutRelHyp}). Each vertex stabilizer $H=G_v$ has a natural peripheral structure $\mathcal{P}_H$ described above, which we denote here by $\mathcal{P}_v$. Adding the finite and $2$--ended nonparabolic groups that stabilize edges incident to $v$ produces a new relatively hyperbolic structure $\mathcal{Q}_v$ by Remark~\ref{rem:ChangingPeripheral}.  Since the Bowditch boundary of $(G_v, \mathcal{Q}_v)$ is obtained from the boundary of $(G_v, \mathcal{P}_v)$ by pinching, we call this new peripheral structure the \emph{pinched peripheral structure} of $G_v$. 
\end{defn}

\begin{defn}[Quadratically hanging]
\label{def:QH}
A vertex stabilizer $G_v$ of such a splitting is \emph{quadratically hanging} if it is an extension
\[
   1 \to F \to G_v \to \pi_1(\Sigma) \to 1,
\]
where $\Sigma$ is a complete, finite area hyperbolic $2$--orbifold (possibly with geodesic boundary and cusps) and $F$ is an arbitrary finite group called the \emph{fiber}.  We also require that each peripheral subgroup of the pinched peripheral structure $\mathcal{Q}_v$ is contained in the pre-image in $G_v$ of a boundary or cusp subgroup of $\pi_1(\Sigma)$.
\end{defn}

\begin{lem}
\label{lem:CircleBoundary}
Let $(G,\mathcal{P})$ be relatively hyperbolic with connected boundary.
Let $G_v$ be a vertex stabilizer of an elementary splitting relative to $\mathcal{P}$.
Then $G_v$ is quadratically hanging with finite fiber if and only if the pinched boundary $\boundary(G_v,\mathcal{Q}_v)$ is homeomorphic to a circle $S^1$.
\end{lem}

\begin{proof}
By Theorem~\ref{thm:Connected}, the group $G$ does not split relative to $\mathcal{P}$ over any finite subgroup.  Therefore, if $G_v$ is quadratically hanging with finite fiber, every boundary component of $\Sigma$ is used;
in other words, by Guirardel--Levitt \cite[Lem.~5.16]{GuirardelLevitt} the lift to $G_v$ of each boundary subgroup of $\pi_1(\Sigma)$ is parabolic in the pinched peripheral structure $\mathcal{Q}_v$.
It follows that with this peripheral structure, the Fuchsian group $\pi_1(\Sigma)$ has finite covolume, so the boundary $\boundary (G_v,\mathcal{Q}_v)$ is homeomorphic to $S^1$.

For the converse, assume the pinched boundary is a circle.  Then $G_v$ has a geometrically finite convergence group action on $S^1$ such that $\mathcal{Q}_v$ is a set of representatives of the maximal parabolic subgroups.
If $F$ is the finite kernel of the action, then $G_v/F$ is a faithful convergence group on $S^1$ with limit set $S^1$.
It follows from \cite{Tukia_Fuchsian,Gabai92,CassonJungreis94} that the action on $S^1$ extends to an isometric action on $\Hyp^2$. Thus, $G_v/F$ is a geometrically finite Fuchsian group of the first kind; \emph{i.e.}, it acts on $\Hyp^2$ with finite covolume so that the subgroups in $\mathcal{Q}_v$ are parabolic (see, for instance, \cite[Chap.~10]{Beardon83_DiscreteGroups}).
\end{proof}

\begin{defn}[Rigid]
\label{def:Rigid}
A relatively hyperbolic pair $(G,\mathcal{P})$ is \emph{rigid} if $G$ has no elementary splittings relative to $\mathcal{P}$.

If $(G,\mathcal{P})$ is relatively hyperbolic, a vertex stabilizer $G_v$ of an elementary splitting relative to $\mathcal{P}$ is called \emph{rigid} if 
$(G_v,\mathcal{Q}_v)$ is rigid in the above sense, where $\mathcal{Q}_v$ is the corresponding pinched peripheral structure.
\end{defn}

Let $M$ be a Peano continuum.
A \emph{cut pair} is a pair of distinct points $\{x,y\}$ in $M$ such that $M\setminus\{x,y\}$ is disconnected, but neither $x$ nor $y$ is a cut point of $M$.
A cut pair $\{x,y\}$ is \emph{inseparable} if its points are not separated by any other cut pair. 
Let $x$ be a local cut point that is not a cut point of $M$. The \emph{valence} of $x$ in $M$ is the number of ends of $M \setminus x$. Such a cut pair $\{x,y\}$ in $M$  is \emph{exact} if the number of components of $M \setminus \{x,y\}$ is equal to the valence of both $x$ and $y$. 

\begin{thm}[\cite{GuirardelLevitt_Canonical,HaulmarkHruska_JSJ}]
\label{thm:JSJ} 
Let $(G,\mathcal{P})$ be relatively hyperbolic with $M=\boundary (G,\mathcal{P})$ connected.
For any elementary splitting of $G$ relative to $\mathcal{P}$, the vertex and edge stabilizers are relatively quasiconvex.
There exists an elementary splitting of $G$ relative to $\mathcal{P}$, called the \emph{JSJ decomposition}, such that each vertex stabilizer is exactly one of the following types:
\begin{enumerate}
    \item a nonparabolic maximal $2$--ended subgroup whose limit set is an exact inseparable cut pair of $M$,
    \item a peripheral subgroup whose limit set is a single point that is a cut point of $M$,
    \item a quadratically hanging subgroup with finite fiber, or
    \item a rigid subgroup whose limit set is not separated by any cut point or exact cut pair of $M$.
\end{enumerate}
The JSJ decomposition $T$ is canonically determined by the topology of $M$.  Every homeomorphism of $M$ induces a type-preserving automorphism of $T$.
\end{thm}

Vertex stabilizers that are quadratically hanging with finite fiber can also be rigid according to the definitions above (see, for example, Guirardel--Levitt \cite[\S 5]{GuirardelLevitt}). However, if a rigid subgroup $G_v$ has the property that its limit set is not separated by any exact cut pair, then every finite-index subgroup of $G_v$ also has this property.  Therefore, any quadratically hanging subgroup has its limit set separated by an exact cut pair. 

\section{Proper group actions on surfaces}
\label{sec:ProperActions}

The action of a relatively hyperbolic group on its Bowditch boundary is a topological action.
It is well known that the theory of topological surfaces is essentially equivalent to the theory of smooth surfaces.  In particular, a folk theorem states that a topological surface does not admit ``wild'' proper actions by homeomorphisms.
We sketch a proof of this result below for the sake of completeness.

\begin{thm}
\label{thm:GeometricOrbifold}
Suppose a group $G$ acts properly by homeomorphisms on a connected surface $X$.
Then $X$ admits a complete metric of constant curvature modeled on either $S^2$, $\E^2$, or $\Hyp^2$ such that the action is isometric.
\end{thm}

By convention, all surfaces here are second countable and Hausdorff. An action is \emph{proper} if each compact set meets only finitely many of its translates.

\begin{proof}
It suffices to prove the result when $G$ acts faithfully.
Each $x \in X$ has a neighborhood basis of $\Stab(x)$--invariant closed discs (see, for instance, \cite{Kolev06}).
By properness, the quotient $X/G$ is Hausdorff and for each $x$ we may choose such an invariant disc $\Delta$ so that $\Delta \cap g\Delta$ is nonempty only when $g \in \Stab(x)$.
A theorem of Ker\'{e}\-kj\'{a}r\-t\'{o} implies that the action of $\Stab(x)$ on $\Delta$ is topologically conjugate to an orthogonal action \cite{Kolev06}.
By a theorem of Newman (see \cite{Dress_NewmanThm}), the fixed set of the finite group $\Stab(x)$ is nowhere dense, so the action of $\Stab(x)$ on $\Delta$ is faithful.
Thus, $X/G$ is a locally linear good orbifold whose underlying space is a surface.

First consider the case that $X/G$ is orientable.
Every orientable triangulated surface admits a complex structure by a construction of Ahlfors--Sario \cite[II.5E]{AhlforsSario_RiemannSurfaces}.
A complex structure on $X/G$ lifts via the branched covering to a complex structure on the universal cover $\tilde{X}$ of $X$ as in \cite[II.4B]{AhlforsSario_RiemannSurfaces}.
The existence of a geometric structure then follows from the Uniformization Theorem in the usual way.
The argument is analogous in the nonorientable case using orbifold coverings, Riemann surfaces without orientation, and the class of directly and indirectly conformal mappings.
\end{proof}

Using a recent result of \Haissinsky--Lecuire, one concludes that any group acting properly---but not necessarily faithfully---on the plane is a virtual surface group, as follows.

\begin{cor}
\label{cor:VirtuallyTorsionFree}
If $G$ is a finitely generated group acting properly on $\R^2$ then $G$ is linear and has a finite-index surface subgroup.
\end{cor}

\begin{proof}
If the action of $G$ on $\R^2$ is faithful, then $G$ is a finitely generated linear group, and the result follows from Selberg's Lemma.
The general case is a consequence of the following result \cite[Thm.~1.3]{petercyril}:
Let $G$ be a group extension with finite kernel and with quotient the fundamental group of a compact surface. Then $G$ is linear and has a finite-index torsion-free subgroup.
\end{proof}

\section{Examples}
\label{sec:Examples}

This section focuses on examples illustrating two phenomena which are relevant for the hypotheses in our main theorem and in our conjecture.  The first is that there exist relatively hyperbolic group pairs $(G,\mathcal{P})$ with planar, connected boundary $M=\partial(G, \mathcal{P})$ such that $G$ is not virtually the fundamental group of a $3$--manifold.  All such known examples occur when $M$ has cut points, in other words $(G,\mathcal{P})$ admits a nontrivial peripheral splitting.

The second phenomenon of interest is that the Bowditch boundary may contain parabolic cut pairs, even when $(G, \mathcal{P})$ does not split over a $2$--ended group relative to $\mathcal{P}$. A \emph{parabolic cut pair} is a cut pair consisting of two parabolic points. In Corollary~\ref{cor:NoParabolicExact}, we show that if the boundary is connected with no cut points, then inseparable parabolic cut pairs can only exist if $G$ splits over a finite group. 

\begin{defn}[Trees of circles]
Let $M$ be a Peano continuum.  A subset $C$ is a \emph{cyclic element} if $C$ consists of a single cut point or contains a non-cutpoint $p$ and all points $q$ that are not separated from $p$ by any cut point of $M$.
Each Peano continuum is the union of its cyclic elements, and each pair of cyclic elements intersects in at most one point that is a cut point of $M$ (see Wilder \cite[\S III.3]{Wilder_Topology}).
A \emph{tree of circles} is a Peano continuum whose nontrivial cyclic elements are homeomorphic to the circle $S^1$.
Any tree of circles admits a planar embedding by a classical theorem of Ayres \cite{Ayres29}.
\end{defn}

\begin{prop}[Hide stuff in the peripheral]
\label{prop:HidingStuff}
For any finitely generated group $P$ with an infinite order element, there exists a relatively hyperbolic group pair $(G,\mathcal{P})$ with each peripheral subgroup isomorphic to $P$ and with Bowditch boundary planar and homeomorphic to a tree of circles.

In particular, if $P$ is not virtually the fundamental group of a $3$--manifold, then neither is $G$.
\end{prop}

For example, one could choose $P$ to be any group that does not coarsely embed in $\R^3$ or any incoherent group.

\begin{proof}
The proof is an elaboration of a simple idea due to Dahmani \cite[Prop.~2.1]{Dahmaniparabolic}.
Suppose $P$ contains an infinite cyclic subgroup $Q$.
Let $F$ be the fundamental group of a torus with one boundary component. Consider the free product with amalgamation $G= F *_{\Z} P$, where the copy of $\Z$ in $F$ corresponds to the boundary curve and the copy of $\Z$ in $P$ is the subgroup $Q$.
By Dahmani's combination theorem \cite{Dahmani03Combination} the group pair $(G, \mathcal{P})$ is relatively hyperbolic, where $\mathcal{P}$ is the set of all conjugates of $P$ in $G$, and the boundary $\boundary (G,\mathcal{P})$ is a tree of circles, since $\boundary (F,\Z) = S^1$.
In particular, the boundary is planar.
\end{proof} 

In the previous result, the peripheral subgroup is the obstruction to being a $3$--manifold group.  In the following result, we show that a relatively hyperbolic group can fail to be a virtual $3$--manifold group, even when all peripheral subgroups are virtually abelian.

\begin{prop}[Three slopes in the torus]
\label{prop:ThreeSlopes}
There exists a group that is hyperbolic relative to 
free abelian groups of rank two with a planar Bowditch boundary homeomorphic to a tree of circles but which is not virtually the fundamental group of any $3$--manifold.
\end{prop} 

\begin{proof}
Let $T^2$ be a $2$--dimensional flat torus, and let $a$, $b$, and $c$ be three essential simple closed geodesics with slopes $0$, $1$, and $\infty$ with respect to a standard basis for $\Z^2$. Let $F_a$, $F_b$, and $F_c$ be three orientable hyperbolic surfaces each of genus one and each with one geodesic boundary component.
Form a locally $\CAT(0)$ space $X$ from the torus $T^2$ by gluing the boundary curve of each surface $F_a$, $F_b$ and $F_c$  to the curves $a$, $b$, and $c$ respectively. We assume that the initial metrics are chosen so that the lengths of glued curves agree.
Then the fundamental group $G=\pi_1(X)$ naturally splits as a graph of groups with four vertex groups, corresponding to the given decomposition of $X$.
The universal cover $\tilde{X}$ is a $\CAT(0)$ space with isolated flats on which $G$ acts properly, cocompactly, and isometrically.
But the visual boundary $\boundary \tilde X$ of this $\CAT(0)$ space contains an embedded copy of $K_{3,3}$.  Indeed, there is a $K_{3,3}$ consisting of a circle that is the boundary of a flat $\tilde{T}^2$ and arcs determined by the surfaces $F_a$, $F_b$, and $F_c$ which connect the endpoints of $a$, $b$, and $c$ on this circle. By the path-connectedness theorem of Ben-Zvi \cite{Benzvi_PathCon}, such paths exist in the complement of the given circle, even though the visual boundary of $\tilde{X}$ is not locally connected by \cite{MihalikRuane99}.

Therefore, $G$ does not coarsely embed in any contractible $3$--manifold by \cite{BestvinaKapovichKleiner02}.
Since $G$ is one ended, it follows that $G$ does not contain a finite-index subgroup that is the fundamental group of a $3$--manifold.

Now we let $\mathcal{P}$ be the set of all conjugates of $\pi_1(T^2)$. 
The group pair $(G,\mathcal{P})$ is relatively hyperbolic by Dahmani's combination theorem \cite{Dahmani03Combination}.
The Bowditch boundary $\boundary(G,\mathcal{P})$, which is not the same as the visual boundary $\boundary X$ by \cite{Tran13}, is again a tree of circles as above.
\end{proof}

Because the examples above have boundaries with cut points, we will mainly examine relatively hyperbolic groups whose boundaries have no cut points.
The exact cut pairs, which are the endpoints of loxodromic axes, are well understood and correspond to splittings over two-ended subgroups.  However, cut pairs where both local cut points are parabolic are not well understood.  These examples were new to us, so we include them here.  We show in Theorem~\ref{thm:nocutpairs} that in any rigid relatively hyperbolic group pair (no elementary splittings relative to $\mathcal{P}$) with boundary not a circle, the existence of a parabolic cut pair in $\partial(G, \mathcal{P})$ implies that $G$ splits over a finite group.

\begin{prop}
\label{prop:ParabolicCutPairs}
There exists a relatively hyperbolic group pair $(G, \mathcal{P})$ that is rigid in the sense that $G$ does not split over any elementary subgroup relative to $\mathcal{P}$  and such that $\partial(G, \mathcal{P})$ has parabolic cut pairs.
\end{prop}

\begin{proof}[First proof]
The first example acts on a hyperbolic building and was constructed by Gaboriau--Paulin (see \cite{GaboriauPaulin01}, \S 3.4, Example~1). 
Let $G=A*B*C$ for $A=B=C=\Z/3\Z$, and let $\mathcal{P}$ be the set of all conjugates of the subgroups $A*B$, $B*C$, and $A*C$.  We note that $G$ is virtually free.  In particular, $G$ is not one ended.
On the other hand, $G$ does not split relative to $\mathcal{P}$ by Serre's Lemma: if $G$ is generated by a finite set $\{s_i\}$ and acts on a tree such that each $s_i$ and each $s_is_j$ has a fixed point, then $G$ has a fixed point \cite[\S I.6.5, Cor.~2]{Serre77}.
Gaboriau--Paulin show that $(G,\mathcal{P})$ is relatively hyperbolic.  We review their construction so that we may examine the associated Bowditch boundary.

Let $T_{3,3}$ be a bipartite tree with vertex set $\mathcal{V} \sqcup \mathcal{W}$ such that all vertices of either type have valence $3$.
Form a $\CAT(-1)$ space $X$ as follows: Start with one ideal triangle of $\Hyp^2$ for each vertex of $\mathcal{V}$ and one copy of the real line for each vertex of $\mathcal{W}$.
For each vertex of $\mathcal{V}$ identify the three sides of the corresponding ideal triangle isometrically with the three adjacent lines.
We choose these isometries so that, for each vertex of $\mathcal{W}$ corresponding to a line $\ell$, the union of the three triangles glued along $\ell$ admits an isometry of order three fixing $\ell$ pointwise that cyclically permutes the adjacent triangles.

Observe that $G=A*B*C$ acts properly and isometrically on $X$ with quotient a single ideal triangle. (The quotient object can be viewed as a complex of groups with $\Z/3\Z$ labels on the edges.) The stabilizers of the lines of $\mathcal{W}$ are the conjugates of $A$, $B$, and $C$, while each ideal triangle of $\mathcal{V}$ has trivial stabilizer.
By \cite{BowditchRelHyp}, the action of $G$ on $\boundary X$ is a geometrically finite convergence group action.
The maximal parabolic subgroups of this action are the family $\mathcal{P}$ of conjugates of $A*B$, \ $B*C$, and $A*C$.  Cutting $X$ along any line $\ell$ of $\mathcal{W}$ splits $X$ into three pieces.  Therefore the two parabolic points at the ends of $\ell$ form a cut pair in the boundary $\boundary X = \boundary(G,\mathcal{P})$. 
\end{proof}

We now discuss a different construction that is much more flexible and shows that groups with parabolic cut pairs are abundant.

\begin{proof}[Second proof]
Start with two one-ended rigid hyperbolic groups $G_1$ and $G_2$ (not splitting over a virtually cyclic subgroup).  Let $A_i$ and $B_i$ be distinct infinite cyclic subgroups of $G_i$ that are each maximal $2$--ended in $G_i$.  Let $G = G_1*G_2$ and let $\mathcal{P}$ consist of all conjugates of the subgroups $A_1 * A_2$ and $B_1 * B_2$. Then $(G,\mathcal{P})$ is relatively hyperbolic by the Bigdely--Wise combination theorem \cite{BigdelyWise_Combination}. 

We will show that its boundary has parabolic cut pairs by examining the geometry of an associated cusped space.
Let $(X_i,x_i)$ be a finite $2$--complex with basepoint such that $\pi_1(X_i,x_i)=G_i$. Glue an interval $I$ from $x_1 $ to $x_2$.   The resulting space $X$ has fundamental group $G$.
The cusped space $Y$ of Groves--Manning \cite{GrovesManning08DehnFilling} is formed from the universal cover $\tilde{X}$ by gluing combinatorial horoballs along the left cosets of $A_1*A_2$ and $B_1*B_2$.
Each lift of $I$ to $\tilde{X}$ is an interval that separates $\tilde{X}$.  But in $Y$ exactly two horoballs have been attached along this interval, one corresponding to a conjugate of $A_1*A_2$ and one corresponding to a conjugate by the same element of $B_1*B_2$. Thus the pair of parabolic points corresponding to those peripheral subgroups disconnects the boundary. (The pair is the limit set of the suspension of a conjugate of $I$ which disconnects $Y$.) 
\end{proof}


\section{A simplicial inseparable cut pair tree}
\label{sec:Peano}

This section explores properties of Peano continua without cut points and the structure of their cut pairs.  
We prove Proposition~\ref{prop:simplicialtree}, which provides a simplicial tree dual to the set of all inseparable cut pairs. 
The structure of cut pairs and other finite separating sets of a Peano continuum is studied in much greater generality in Papasoglu--Swenson \cite[Thm.~6.6]{PapasogluSwenson11_Cactus}.  A related result is also proved by Guralnik in \cite[Thm.~3.15]{Guralnik05}.
The proof here is self-contained and significantly shorter than the  proof of \cite[Thm.~6.6]{PapasogluSwenson11_Cactus}, as it is tailored to the specific case of cut pairs. 

Recall that a \emph{Peano continuum} is a compact, connected, locally connected, metrizable space. Several well-known properties of Peano continua are summarized in the following remark (see Wilder \cite{Wilder_Topology} for proofs).

\begin{rem}[Properties of Peano continua]
A Peano continuum $M$ is \emph{uniformly locally connected} in the following sense:
for each $\epsilon>0$ there exists $\delta>0$ such that any two points of $M$ with distance less than $\delta$ are contained in a connected subset of $M$ of diameter less than $\epsilon$.

In a locally connected space, the components of any open set are open.  Moreover, every connected open subset $U$ of a Peano continuum $M$ is arcwise connected.  A closed set $S\subset M$ \emph{separates} points $a$ and $b$ of $M\setminus S$ if $a$ and $b$ are in different components of $M\setminus S$.
It follows that a closed set $S$ separates $a$ from $b$ if and only if every path in $M$ from $a$ to $b$ intersects $S$. 
\end{rem}

Let $M$ be a Peano continuum.
A \emph{cut pair} is a pair of distinct points $\{x,y\}$ in $M$ such that $M\setminus\{x,y\}$ is disconnected, but neither $x$ nor $y$ is a cut point of $M$.
A cut pair $\{x,y\}$ is \emph{inseparable} if its points are not separated by any other cut pair.

The following definition was extensively studied by Bowditch \cite{Bowditch99Treelike}.

\begin{defn}[Pretrees]
Let $\mathcal{V}$ be a set with a ternary relation $\mathcal{R} \subset \mathcal{V}\times \mathcal{V} \times \mathcal{V}$. If $(A,B,C)\in \mathcal{R}$ we say that ``$B$ is between $A$ and $C$.''  If $A,B\in \mathcal{V}$, the \emph{open interval} $(A,B)$ between $A$ and $B$ is the set of all members of $\mathcal{V}$ lying between $A$ and $B$. The \emph{closed interval} $[A,B]$ is the set $(A,B) \cup \{A,B\}$.

We say that $(\mathcal{V},\mathcal{R})$ is a \emph{pretree} if it satisfies the following four conditions:
\begin{enumerate}
    \item
    \label{item:Reflexive}
    $[A,A] = \{A\}$,
    \item
    \label{item:Symmetric}
    $[A,B] = [B,A]$,
    \item
    \label{item:Monotone}
    If $B \in (A,C)$ then $C \notin (A,B)$, and
    \item
    \label{item:Triangle}
    $[A,C] \subseteq [A,B] \cup [B,C]$.
\end{enumerate}
\end{defn}

Betweenness has been studied in various settings.
The following notion of betweenness for inseparable cut pairs that are not necessarily disjoint was introduced by Papasoglu--Swenson \cite{PapasogluSwenson06} (\emph{cf.} Guralnik \cite{Guralnik05}).

\begin{defn}[Betweenness]
Let $M$ be a Peano continuum.
Let $\mathcal{V}_I$ be the set of all inseparable cut pairs of $M$.
We define a ``betweenness'' ternary relation on $\mathcal{V}_I$ as follows:
An inseparable cut pair $C$ is \emph{between} inseparable cut pairs $A$ and $B$ if the set $C$ separates at least one point of $A$ from at least one point of $B$.
We note that, by inseparability, $C$ separates all points of $A\setminus C$ from all points of $B\setminus C$.  
Equivalently $C$ is between $A$ and $B$ if and only if every path from $A\setminus C$ to $B\setminus C$ intersects $C$.
\end{defn}

\begin{lem}
For any Peano continuum $M$, the betweenness relation defined above gives the set of inseparable cut pairs $\mathcal{V}_I$ the structure of a pretree.
\end{lem}

\begin{proof}
Conditions (\ref{item:Reflexive}) and~(\ref{item:Symmetric}) are immediate from the definition. To see (\ref{item:Monotone}), suppose $B \in (A,C)$.  Then the connected component $U$ of $M\setminus B$ containing $A\setminus B$ is disjoint from $C$.  Note that each point of $B$ lies in the closure $\bar{U}$, since otherwise $B$ would contain a  cut point of $M$.  Therefore $U \cup (B\setminus C)$ is a connected set in the complement of $C$ that intersects both $A$ and $B$.  In particular $C \notin (A,B)$.

For (\ref{item:Triangle}), suppose $D \in (A,C)$ and let $B$ be any inseparable cut pair.
Suppose by way of contradiction that there exist paths $c_1$ from $A\setminus D$ to $B\setminus D$ and $c_2$ from $B\setminus D$ to $C\setminus D$ that both avoid $D$.
If $c_1(1) = c_2(0)$ then the concatenation $c_1 c_2$ provides a contradiction.
If not then $B\setminus D$ contains two distinct points; \emph{i.e.}, the pairs $B$ and $D$ are disjoint.  Since $B$ is inseparable there is a path $c_3$ from $c_1(1)$ to $c_2(0)$ in the complement of $D$. The concatenation $c_1 c_3 c_2$ provides the desired contradiction.
\end{proof}

As explained by Bowditch \cite{Bowditch99Treelike}, in a pretree each interval $(A,B)$ is linearly ordered by the \emph{separation order} given by $X<Y$ if $Y \in (X,B)$.

\begin{lem} \label{lem:between} 
Let $M$ be a Peano continuum, and $A,B$ be inseparable cut pairs in $M$.  Suppose $X$, $Y$, and $Z$ are three inseparable cut pairs in $(A,B)$. If $X< Y<Z$, then $Y$ is between $X$ and $Z$. 
\end{lem} 

\begin{proof}
Since $X < Y < Z$, every path from $X$ to $B$ meets $Y$, and every path from $Y$ to $B$ meets $Z$.  We claim that every path from $X$ to $Z$ meets $Y$.  Indeed, since $Z$ is not less than $Y$, there is a path $p$ from $Z$ to $B$ that does not meet $Y$. Let $q$ be any path from $X$ to $Z$.  The path $q$ followed by the path $p$ is a path from $X$ to $B$.  Since $X <Y$, this path meets $Y$.  Since $p$ does not meet $Y$, the arbitrary path $q$ from $X$ to $Z$ meets $Y$.
\end{proof} 
  
We say that a sequence of cut pairs $(X_i)$ \emph{converges to a point} $x_\infty$ if every neighborhood of $x_\infty$ contains all but finitely many of the cut pairs $X_i$.
  
\begin{lem}
\label{lem:monotonic} 
Let $M$ be a Peano continuum without cut points, and let $A,B \subset M$ be inseparable cut pairs.
Suppose $X_1 < X_2 < \cdots < X_i < \cdots$ is a monotonic sequence of inseparable cut pairs in the interval $(A,B)$.
Then $\diam(X_i) \to 0$, and the sets $X_i$ have a subsequence that converges to a point.
\end{lem} 

\begin{proof}
Suppose by way of contradiction that $\diam(X_i)$ does not limit to zero.  Then after passing to a subsequence, we may assume that $\diam(X_i)$ is bounded away from zero.  In this circumstance, we will show that $M$ is not locally connected, contradicting the fact that $M$ is a Peano continuum.

Let $\overrightarrow{I}$ be closure of the component of $M \setminus X_I$ which contains all the $X_j$ with $j>I$. Note that a component containing all such $X_j$ must exist, since the given sequence of cut pairs is a nested sequence. Similarly, let $\overleftarrow{I}$ be the closure of the component of $M \setminus X_I$ that contains all $X_j$ with $j<I$. Let $U_I$ be the intersection
\[
   U_I = \overrightarrow{I-1} \cap \overleftarrow{I+1},
\]
which is a closed set containing $X_I$.    We claim that 
$X_I= \{x_I,y_I\}$ is contained in a single component of $U_I$.   If not, then every connected set containing $X_I$ either meets both points of $X_{I-1}$ or meets both points of $X_{I+1}$.  But then $\{x_{I-1},x_{I+1}\}$ is a cut pair that separates $x_I$ from $y_I$, contradicting inseparability of the cut pair $X_I$.  

Suppose $\diam(X_i)$ is bounded below by a positive number $c$ as $i \to \infty$.
Let $(a_i)$ be a sequence of points in $M$ such that 
 \begin{enumerate} 
 \item $a_i \in U_i$ 
 \item $d(a_i, x_i) > c/4 $ and $d(a_i, y_i) > c/4 $
  \end{enumerate} 
Such points $a_i$ must exist once $i$ is sufficiently large, since $d(x_i,y_i) \ge d$ and there is a connected subset of $U_i$ containing $x_i$ and $y_i$.

If the compactum $M$ were locally connected then for each $\epsilon>0$ there would exist $\delta>0$ as in the definition of uniformly locally connected.
Observe that if $i\ne j$, any connected set containing both $a_i$ and $a_j$ must have diameter at least $c/4$.
Therefore for any $\epsilon < c/4$, no corresponding $\delta$ exists, since $(a_i)$ has a Cauchy subsequence.
It follows that $M$ is not locally connected.
Since $M$ is a Peano continuum, we have reached a contradiction.
\end{proof}
 
\begin{lem}
\label{lem:GlobalCutPoint}
Let $M$ be a Peano continuum, and let $A,B \subset M$ be inseparable cut pairs.  Suppose $(X_i)_{i=1}^{\infty}$ is a sequence of inseparable cut pairs contained in the interval $(A,B)$.  If the sets $X_i$ converge to a single point $x_\infty$, then $x_\infty$ is a cut point of $M$.
\end{lem} 
 
\begin{proof}
Let $\overrightarrow{A}$ be the closure of the component of $M \setminus A$ containing $B\setminus A$, and let $\overleftarrow{B}$ be the closure of the component of $M\setminus B$ containing $A\setminus B$.
Choose points $u$ and $v$ in the open sets $M \setminus \overrightarrow{A}$ and $M \setminus \overleftarrow{B}$ respectively.
Then every path from $u$ to $v$ intersects both $A$ and $B$.  It follows that every path from $u$ to $v$ intersects each cut pair $X_i$ of the given sequence.
Since a path is a closed set, each path from $u$ to $v$ also must intersect the limit point $x_\infty$.
Observe that $x_\infty$ lies in the closed set $\overrightarrow{A} \cap \overleftarrow{B}$, which contains neither $u$ nor $v$.  Thus $x_\infty$ is distinct from each of $u$ and $v$.
Therefore $x_\infty$ is a cut point of $M$ separating $u$ from $v$.
\end{proof}

\begin{prop}[Discreteness]
\label{prop:Discrete}
Let $M$ be a Peano continuum without cut points. Let $A$ and $B$ be inseparable cut pairs of $M$.
Then the interval $(A,B)$ is finite.
\end{prop}

\begin{proof}
Suppose by way of contradiction that the interval $(A,B)$ contains infinitely many inseparable cut pairs.
By a bisection argument, we will show that there is a monotonic sequence with respect to the order $<$ (or its reverse order obtained by switching the roles of $A$ and $B$).
Choose a cut pair $X_1$ from the interval $(A,B)$.
By Lemma~\ref{lem:between},  one of the intervals $(A,X_1)$ or $(X_1,B)$ also contains infinitely many inseparable cut pairs.  Let $(A_1,B_1)$ denote this new interval.
Continuing recursively, we produce an infinite sequence of distinct intervals $(A_i,B_i)$ each containing infinitely many inseparable cut pairs.
Furthermore these intervals are nested in the sense that
\[
   A_1 \le A_2 \le \cdots \le A_i \le \cdots
\]
and 
\[
   B_1 \ge B_2 \ge \cdots \ge B_i \ge \cdots
\]
Since the intervals $(A_i,B_i)$ for $i=1,2,3,\dots$ are pairwise distinct, there are either infinitely many distinct left endpoints or infinitely many distinct right endpoints.
In either case (possibly by switching the roles of $A$ and $B$) there exists a monotonic sequence of inseparable cut pairs contained in the original interval $(A,B)$. By Lemma~\ref{lem:monotonic}, these cut pairs converge to a single point, which must be a cut point by Lemma~\ref{lem:GlobalCutPoint}, contradicting the assumption that $M$ has no cut points.  Therefore all intervals $(A,B)$ are finite.
\end{proof}

Let $M$ be a Peano continuum without cut points.  
Recall that $\mathcal{V}_I$ is the set of all inseparable cut pairs in $M$.
A \emph{star} is a maximal subset $S\subseteq\mathcal{V}_I$ with the property that for each $A,B \in S$ the interval $(A,B)$ is empty.  Let $\mathcal{W}$ be the set of all stars.
We define a bipartite graph $\mathcal{T}_M$ with vertex set $\mathcal{V}_I \sqcup \mathcal{W}$ such that
two vertices $V \in \mathcal{V}_I$ and $W \in \mathcal{W}$ are joined by an edge in $\mathcal{T}_M$ whenever $V \in W$.

\begin{prop}[Inseparable cut pair tree]
\label{prop:simplicialtree}
Let $M$ be a Peano continuum without cut points.
Then the graph $\mathcal{T}_M$ defined above, the \emph{inseparable cut pair tree}, is a simplicial tree.
\end{prop}

\begin{proof}
By Proposition~\ref{prop:Discrete}, the set of inseparable cut pairs is a discrete pretree in the sense of Bowditch.  In \cite[\S 3]{Bowditch_Peripheral}, Bowditch shows that the above construction produces a simplicial tree when applied to any discrete pretree.
\end{proof}

\section{Inseparable cut pairs are loxodromic} \label{sec:inseplox} 

This section examines the case of a Peano continuum without cut points that arises as a Bowditch boundary of a relatively hyperbolic group.  In this setting, the existence of a convergence group action allows us to establish stronger properties of the inseparable cut pair tree.

The following proposition is the main result of this section and will be used in the proof of the stronger theorem asserting no cut pairs, which is Theorem~\ref{thm:nocutpairs} below. 

\begin{prop}
\label{prop:inseparable} 
Let $G$ be one ended. Suppose $(G,\mathcal{P})$ is relatively hyperbolic and $M=\boundary (G,\mathcal{P})$ is not homeomorphic to the circle $S^1$. 
Assume $(G,\mathcal{P})$ has no elementary splittings relative to $\mathcal{P}$.
Then $M$ is a Peano continuum without cut points and without inseparable cut pairs, such that all local cut points are parabolic.
\end{prop} 

The proof involves combining the following three lemmas.

\begin{lem}
\label{lem:NoSplitting} 
Suppose $(G,\mathcal{P})$ is relatively hyperbolic with boundary $M=\boundary (G,\mathcal{P})$ not homeomorphic to the circle $S^1$. 
Assume $(G,\mathcal{P})$ has no elementary splittings relative to $\mathcal{P}$.
Then $M$ is a Peano continuum without cut points such that all local cut points are parabolic.
\end{lem} 

\begin{proof}
The conclusion that $M$ is a Peano continuum with no cut points is the conclusion of Theorem~\ref{thm:Connected}.
Furthermore, since $M\ne S^1$, all local cut points of $M$ are parabolic by a theorem of Haulmark \cite{HaulmarkRelHyp}.
\end{proof}

Recall the inseparable cut pair tree for $T_M$, defined just above Proposition~\ref{prop:simplicialtree}. 

\begin{lem}
\label{lem:MinimalTree}
Let $(G,\mathcal{P})$ be relatively hyperbolic, and suppose $M=\boundary(M,\mathcal{P})$ is a Peano continuum without cut points.
Then $G$ acts minimally on the inseparable cut pair tree $\mathcal{T}_M$.
\end{lem}

\begin{proof}
We first show that $\mathcal{T}_M$ does not contain vertices of valence one. Recall that $\mathcal{T}_M$ is bipartite with vertex set $\mathcal{V}_I \sqcup \mathcal{W}$.  We first consider the valence of an inseparable-cut-pair vertex $A \in \mathcal{V}_I$.  
The convergence action of $G$ on $M$ is \emph{minimal} in the sense that $M$ does not contain a nonempty $G$--invariant closed proper subset.
Since each component of $M \setminus A$ is open in $M$, each component $U$ contains an inseparable cut pair (even in the orbit of $A$).  
By Proposition~\ref{prop:Discrete}, for each $U$ there is a cut pair $B\subset U$, such that the interval $(A,B)$ is empty.  By Zorn's Lemma, the set $\{A,B\}$ is a subset of a star, \emph{i.e.}, a maximal set of cut pairs. This star is a $\mathcal{W}$--vertex adjacent to the $\mathcal{V}_I$--vertex $A$.  
Indeed, neighbors of $A$ are in one-to-one correspondence with the components of $M \setminus A$. 
Therefore the inseparable-cut-pair vertex $A \in \mathcal{V}_I$ does not have valence one.

The argument in the preceding paragraph implies that each star contains more than one inseparable cut pair. So a star vertex $W \in \mathcal{W}$ also cannot have valence one. 

Now we claim that $G$ acts minimally on $\mathcal{T}_M$. Suppose by way of contradiction that $G$ stabilizes a proper subtree $T'$.  Then there is an edge that separates $T'$ from its complement.  Since there are no valence one vertices, there are vertices of type $\mathcal{V}_I$ on either side of this edge. This edge goes between an inseparable cut pair $A$ an a star $S$ which contains it.  Since the orbit of $A$ contains cut pairs in each component of $M \setminus A$, (since each of these components is open in $M$ and the orbit is dense) neither of the pieces cut off by this edge are $G$--invariant. 
\end{proof}

\begin{lem}
\label{lem:InseparableExact}
Suppose $(G,\mathcal{P})$ is relatively hyperbolic and $M=\boundary(G,\mathcal{P})$ is a Peano continuum without cut points.
If $G$ is one ended, then each inseparable cut pair of $M$ consists of the endpoints of a loxodromic element.  In particular, all inseparable cut pairs are exact.
\end{lem}

\begin{proof}
By Lemma~\ref{lem:MinimalTree}, the action of $G$ on the inseparable cut pair tree $\mathcal{T}_M$ is minimal.
Since $G$ is one ended, it follows from Stallings' Theorem that every edge of $\mathcal{T}_M$ has an infinite stabilizer.  
In particular, each inseparable cut pair $C$ has an infinite stabilizer $H$.
Without loss of generality, assume each point of $C$ is fixed by $H$ (passing to an index two subgroup if necessary).
Every fixed point of $H$ is contained in the limit set $\Lambda(H)$, so that $C \subseteq \Lambda(H)$.
By a result of Tukia \cite[Thm.~2S]{Tukia94}, if a subgroup $H$ of a convergence group has at least one fixed point then $\Lambda(H)$ contains at most two points.  Therefore $C=\Lambda(H)$. By Tukia \cite[Thm.~2R]{Tukia94}, the two points of $C$ are the fixed points of a loxodromic element of $H$, establishing the claim.

That loxodromic cut pairs are exact follows from Haulmark--Hruska \cite[Lem.~4.1]{HaulmarkHruska_JSJ}, which is a minor variation of Bowditch \cite[Lem.~5.6]{Bowditch98JSJ}. 
\end{proof}

\begin{cor}
\label{cor:NoParabolicExact}
Suppose $(G,\mathcal{P})$ is relatively hyperbolic and $M=\boundary(G,\mathcal{P})$ is a Peano continuum without cut points.
If $G$ is one ended, then there are no inseparable parabolic cut pairs. 
\end{cor} 

The proof of Proposition~\ref{prop:inseparable} now follows by combining the lemmas above.

\begin{proof}[Proof of Proposition~\ref{prop:inseparable}]
We have shown in Lemma~\ref{lem:NoSplitting} that $M$ has no cut points and all local cut points are parabolic.
Since $T_M$ is minimal by Lemma~\ref{lem:MinimalTree}, the existence of an inseparable cut pair will imply that $T_M$ is nontrivial.  By Lemma~\ref{lem:InseparableExact}, all inseparable cut pairs of $M$ are loxodromic. 
Then $(G, \mathcal{P})$ splits over a $2$--ended group since it acts on a simplicial tree with loxodromic edge stabilizers. 
Thus there are no inseparable cut pairs.
\end{proof}

\section{One-ended rigid groups have no cut pairs} 
\label{sec:ParabolicCutPairs}

The main purpose of this section is to prove the following theorem. 

\begin{thm}
\label{thm:nocutpairs}
Suppose $G$ is one ended and $(G, \mathcal{P})$ is relatively hyperbolic with no elementary splittings relative to $\mathcal{P}$. 
If $M=\boundary(G,\mathcal{P})$ is not homeomorphic to the circle $S^1$,
then $M$ is a Peano continuum with no cut points and no cut pairs.
\end{thm} 

Since there are no inseparable cut pairs in this situation by Proposition~\ref{prop:inseparable}, we will study the separable cut pairs, which we will see have a natural cyclic order. To describe this cyclic order, we introduce the notion of a cyclic decomposition and a cyclic set.
The structure of cyclic sets discussed in this section closely follows work of Papasoglu--Swenson \cite{PapasogluSwenson06} from the more general setting of continua that are not necessarily locally connected. 
We get slightly stronger conclusions here in the presence of local connectedness.

\begin{defn}
A finite set of local cut points $S_0=\{ s_1,\dots,s_n \}$ with $n\ge 3$ is \emph{cyclic} if there exist closed connected subsets  $M_1,\dots,M_n$ of $M$ such that $M = \bigcup_i M_i$ and $M_i \cap M_{i+1} = s_i$, \ $M_n \cap M_1 = s_n$ and $M_i \cap M_j = \emptyset$ otherwise.
Such a family of sets $M_1,\dots,M_n$ is a \emph{cyclic decomposition} corresponding to the cyclic set $S_0$.
\end{defn} 

\begin{defn}
A \emph{necklace}  is a maximal set $N$ with $|N|>2$ such that every finite subset with more than one point is either a cut pair or a cyclic set. 
\end{defn} 

\begin{defn}
Let $N$ be a necklace.  We define an equivalence class on $M \setminus N$ such that $x \sim_N y$ if $x$ is not separated from $y$ by any cut pair contained in $N$. A $\sim_N$--equivalence class is called a \emph{gap} of $N$.
\end{defn} 

\begin{prop}
\label{prop:SeparableInseparable}
Suppose $(G,\mathcal{P})$ is relatively hyperbolic and $M=\boundary(G,\mathcal{P})$ is a Peano continuum with no cut points that is not homeomorphic to the circle $S^1$.
If $M$ has a cut pair, then $M$ contains an inseparable cut pair.
\end{prop}

The proposition is an immediate consequence of the following lemma.

\begin{lem}
Let $M$ be a Peano continuum without cut points that is not homeomorphic to the circle $S^1$.
Then the following hold.
\begin{enumerate}
    \item Every separable cut pair of $M$ is in a necklace.
    \item Every necklace in $M$ has a gap.
    \item If $G$ is a gap of a necklace $S$ in $M$, then $\overline{G}\cap S$ is an inseparable cut pair.
\end{enumerate}
\end{lem}

\begin{proof}
The first assertion follows from Papasoglu--Swenson \cite[Lems.\ 15 and~17]{PapasogluSwenson06} using Zorn's lemma, as explained in \cite[p.~1769]{PapasogluSwenson06}.
To see the second assertion, observe that a necklace $S$ with no gaps would contain no inseparable cut pairs.
Furthermore, we would then have $M=S$.  By Papasoglu--Swenson \cite[Cor.~21]{PapasogluSwenson06}, it follows that $M$ is homeomorphic to $S^1$.

We now consider the third assertion.  Consider a gap $G$ of the necklace $S$.
Choose a cyclic decomposition $M_1,M_2,M_3$ corresponding to three points $s_1,s_2,s_3$ in $S$.   Label the points so that $G \subseteq M_2$.
Let $X=M\setminus\{s_3\}$.  Form a new space $\hat{M}$ which is a two point compactification of  $X$ by adjoining two new points $a$ and $b$ such that $a$ compactifies $M_1 \setminus\{s_3\} $ and $b$ compactifies $M_3\setminus\{s_3\}$.
Note that every point of $S\setminus\{s_3\}$ is a cut point of $\hat{M}$.
Two points of $\hat{M}$ are separated by a point $s$ of $S\setminus\{s_3\}$ if and only if they are separated in $M$ by the cut pair $\{s,s_3\}$.
Therefore the gap $G$ is an equivalence class of points of $\hat{M}$ not separated by any cut point of $\hat M$ that is also an element of $S\setminus\{s_3\}$.

Note that $\hat{M}$ does not contain an embedded arc that intersects $\bar{G}$ only in its endpoints.  Indeed, if there were such an arc $A$, it contains a point $x$ that is separated from $\hat{M}$ by a point of $S\setminus\{s_3\}$.  All paths from $x$ to $\hat{M}$ must pass through this cut point, contradicting that $A$ is an embedded arc.
Therefore $S\setminus \{s_3\}$ contains unique points $a'$ and $b'$ such that every path from $a$ to $\bar{G}$ enters $\bar{G}$ at the point $a'$ and similarly every path from $b$ to $\bar{G}$ enters $\bar{G}$ at $b'$.  Note that $a'\ne b'$, since if they were equal they would give a cut point of $M$. (It would separate $s_3$ from $\bar G$.)  We claim that $(a',b')$ form an inseparable cut pair of $M$. We first note that they are a cut pair, since any path from $s_3$ to $\bar G$ must pass through either $a'$ or $b'$.  Suppose that $a'$ and $b'$ are separated by some other pair $(x,y)$.   Then by Papasoglu--Swenson \cite[Lem.~15]{PapasogluSwenson06} this cut pair is included in our necklace $S$.  Since $a'$ and $b'$ are in the closure of the gap, there are points of the gap that are separated by $x$ and $y$.  This is a contradiction to the definition of gap.  Note that $a'$ and $b'$ are the only element of $S$ in $\bar G$, as $G$ is contained in some $M_i$ for every cyclic subset. \end{proof}

In the case that $G$ is one-ended, we get the stronger conclusion of \ref{thm:nocutpairs}: 

\begin{proof}[Proof of Theorem~\ref{thm:nocutpairs}]  By Proposition~\ref{prop:inseparable}, we know that $M$ has no cut point and no inseparable cut pairs.   If $M$ had any cut pair, it would contain an inseparable cut pair by Proposition~\ref{prop:SeparableInseparable}, a contradiction.  
\end{proof}

\section{Peripheral Subgroups}
\label{sec:PeripheralSubgroups}

In this section, we prove the following theorem. 

\begin{thm}
\label{thm:VirtualSurfaceGroup}
Suppose the finitely generated group $G$ is one ended, the pair $(G, \mathcal{P})$ is relatively hyperbolic, and $G$ does not split relative to $\mathcal{P}$ over a parabolic subgroup.
If the boundary $M=\boundary(G,\mathcal{P})$ is planar then each member of $\mathcal{P}$ is virtually the fundamental group of a compact surface. 
\end{thm}

The proof of Theorem~\ref{thm:VirtualSurfaceGroup} uses the following result.

\begin{prop}
\label{prop:RigidNoCutPairs}
Let $G$ be finitely generated and one ended.
Suppose $(G,\mathcal{P})$ is relatively hyperbolic and $G$ does not split relative to $\mathcal{P}$ over a parabolic subgroup.
Let $G_Z$ be a nonelementary vertex group of the JSJ decomposition of $(G,\mathcal{P})$ over elementary subgroups.
Let $\mathcal{Q}_Z=\mathcal{P}_Z \cup \mathcal{E}_Z$ be the pinched peripheral structure of $G_Z$ as in Definition~\ref{def:PinchedPeripheral}.
If the pinched boundary $M_Z = \boundary(G_Z,\mathcal{Q}_Z)$ is not homeomorphic to the circle, then $M_Z$ is a Peano continuum with no cut points and no cut pairs.
\end{prop}

We note that $G_Z$ might not be one ended, so  Theorem~\ref{thm:nocutpairs} will not apply in all cases. 
As shown by Proposition~\ref{prop:ParabolicCutPairs}, the conclusion of Theorem~\ref{thm:nocutpairs} need not hold when the rigid group in question is not one ended.
Nevertheless, we show that the pinched boundary of $G_Z$ cannot have cut pairs since it arises as a vertex group of a splitting of $G$.

\begin{proof}
If $G_Z$ is a one-ended group, we apply Theorem \ref{thm:nocutpairs}.  In general, 
assume that the pinched boundary $M_Z$ is not a circle. Then $G_Z$ is a rigid vertex group by Lemma~\ref{lem:CircleBoundary}, so $G_Z$ has no elementary splittings relative to $\mathcal{Q}_Z$.
According to Lemma~\ref{lem:NoSplitting}, the pinched boundary is a Peano continuum without cut points such that all local cut points are parabolic.

By way of contradiction, suppose $M_Z$ contains a cut pair. 
Then by Proposition~\ref{prop:SeparableInseparable}, it contains an inseparable cut pair $\{a,b\}$ such that $a$ and $b$ are parabolic points. 
The strategy is to show that there is also a parabolic inseparable cut pair $\{a',b'\}$ in $M$. This conclusion would contradict Corollary~\ref{cor:NoParabolicExact} since $G$ is one ended and $M$ is a Peano continuum without cut points by Theorem~\ref{thm:Connected}.

Observe that the pinched boundary $M_Z$ can be obtained from the connected space $M$ by collapsing to a point each component of the closure of $M \setminus \boundary(G_Z,\mathcal{P}_{Z})$. One component will be collapsed for each edge adjacent to $Z$.  
The proof has two cases, depending on whether the map $M \rightarrow M_Z$ is injective on the preimage of $\lbrace a,b \rbrace$ or not. 

If each of $a$ and $b$ have exactly one preimage in $M$ then the preimage of $ \lbrace a, b \rbrace$ is a cut pair of $M$ consisting of two parabolic points $a'$ and $b'$.   We claim that this cut pair is inseparable in $M$. Note that no cut pair in $M$ has the property that one point is in the limit set of $G_Z$ and the other is not, by inseparability of the limit sets of the edge stabilizers.  Let $\lbrace c', d' \rbrace$ be any other cut pair of $M$.  We claim it does not link with $\lbrace a', b' \rbrace$. If the pair $\lbrace c',d' \rbrace$ is inseparable then it does not link with $\lbrace a', b' \rbrace$, so we assume $\lbrace c',d' \rbrace$ is separable.  If $\lbrace c',d' \rbrace$ in the same vertex stabilizer, it is not the limit set of any edge stabilizer, so the image $M \to M_Z$ is injective on the pair $\{c',d'\}$ and maps it to a cut pair in $M_Z$ which links with $\lbrace a,b \rbrace$, contradicting that $\{a,b\}$ is inseparable.
Therefore $\lbrace c',d' \rbrace$ is not in the limit set of $G_Z$ in $M$.  Since it is not contained in the limit set of $G_Z$, this pair $\{ c',d' \}$ must be separated from the limit set of $G_Z$ by an inseparable cut pair.
But then $\{a',b' \}$ does not link with $\{ c',d' \}$.  
Since $\{a',b'\}$ does not link with any other cut pair, it is inseparable, so it is also parabolic.  We are done in this case. 

Now assume that one of $a$ or $b$, say $a$, was obtained by collapsing the limit set of some edge group in $\mathcal{E}_Z$.  Then $a$ has valence $2$, and $M_Z \setminus \{a\}$ has two ends.  Since $M_Z$ has no cut point, each cut point of $M_Z \setminus \{a\}$ separates these two ends, and the two-point end compactification of  $M_Z \setminus\{a\}$ has a linear separation ordering on its cut points.  
Therefore since the $G_Z$--stabilizer of $a$ acts cocompactly on $M_Z \setminus \{a\}$, there exists a cut point of $M_Z \setminus \{a\}$ to the left of $b$ and one to the right of $b$.  These two points form a cut pair of $M_Z$ which separates $a$ from $b$. So $\{ a,b \}$ is not an inseparable cut pair, contradicting our hypothesis.
It follows that $M_Z$ has no cut pairs.
\end{proof}

\begin{lem}
\label{lem:MZ:Planar}
Let $G$ and $G_Z$ be as in Proposition~\ref{prop:RigidNoCutPairs}.
If the boundary $M$ is planar then so is the pinched boundary $M_Z$.
\end{lem} 

\begin{proof}
We follow a strategy similar to \Haissinsky's proof of \cite[Lem.~6.5]{Haissinsky_Invent}.  We produce an embedding $M_Z \to S^2$ by composing the given embedding $\boundary(G_Z,\mathcal{P}_Z) \to S^2$ with a quotient $S^2\to S^2$ obtained using Moore's Theorem:
If $\mathcal{A}$ is a null family of pairwise disjoint closed, connected, nonseparating sets of $S^2$, then the quotient $S^2/\mathcal{A}$ formed by collapsing each member of $\mathcal{A}$ to a point is homeomorphic to $S^2$ (see \cite[\S 61.IV]{Kuratowski_VolII}).
A family of subsets is \emph{null} if for each $\epsilon>0$ only finitely many members of the family have diameter greater than $\epsilon$.

For each cut pair associated to an edge emanating from $Z$, construct an embedded arc connecting the endpoints of the associated loxodromic in $K= \partial(G_Z, P_Z)$ in a path-connected complementary component of $K$. Call this collection of arcs $\mathcal{A}$. Note that $K$ separates $M$ and the complementary components are attached along cut pairs.  The set of these components is a null family by Lemma~\ref{lem:NullComponents}, proved below. The union of $K$ with this collection of arcs is planar, as the union embeds in $M$.  Now let $q\colon S^2 \rightarrow S^2$ be the quotient obtained by collapsing each arc in this collection to a point. The image of $K$ under this quotient is an embedded copy of $M_Z$ in $S^2$.
\end{proof}

\begin{lem}
\label{lem:NullComponents}
Let $M$ be a Peano continuum, and let $K$ be a compact subset such that for each component $U$ of $M\setminus K$ the frontier $\bar{U} \setminus U$ contains exactly two points.
Then the closures $\bar{U}$ of components of $M\setminus K$ are a null family of Peano subcontinua.
\end{lem}

\begin{proof}
Since each component $U$ is locally connected and has a discrete frontier, $\bar{U}$ is also locally connected.
If the family of components is not null, there exists $\epsilon>0$ and an infinite family of components $U$ of $M\setminus K$ each containing a point $p_U$ with $d(p_U,K)>\epsilon$.  If $U$ and $U'$ are two such components, any connected set containing $p_U$ and $p_{U'}$ has diameter at least $\epsilon$.
However, by compactness, the distance $d(p_U,p_{U'})$ may be chosen arbitrarily close to zero, contradicting the local connectedness of $M$.
\end{proof}

The following result characterizes the complementary components of certain planar Peano continua.

\begin{thm}
\label{thm:PlanarPeano}
Let $M\subset S^2$ be a nontrivial planar Peano continuum.  Then we have the following:
\begin{enumerate}
   \item
   \label{item:Nullity}
   The components of $S^2 \setminus M$ are a null family.
   \item
   \label{item:Torhorst}
   Suppose $M$ has no cut points. For each component $U$ of $S^2 \setminus M$, the boundary $\boundary U$ is a Jordan curve and the closure $\bar{U}$ is a closed disc.
   \item
   \label{item:NoCutPairs}
   Suppose $M$ has no cut points and no cut pairs.
   Let $U_1$ and $U_2$ be components of $S^1 \setminus M$. Then the Jordan curves $\boundary U_1$ and $\boundary U_2$ intersect in at most one point.
\end{enumerate}
\end{thm}

The first two assertions are classical results of planar topology.
Assertion~(\ref{item:Nullity}) is due to Sch\"{o}nflies (see \cite[IV.7.7]{Wilder_Topology}).
Assertion~(\ref{item:Torhorst}) is a result of Torhorst; see \cite[IV.6.12]{Wilder_Topology} for a topological proof and \cite[\S 17]{Milnor_Dynamics} for a complex analytic proof. We provide a short proof of Assertion~(\ref{item:NoCutPairs}).

\begin{proof}[Proof of Theorem~\ref{thm:PlanarPeano}(\ref{item:NoCutPairs})]
Suppose by way of contradiction that $\overline{U}_1 \cap \overline{U}_2$ contains distinct points $x \ne y$. 
For each $i=1,2$ let $c_i$ be a properly embedded arc in $\overline{U}_i$ joining $x$ and $y$; in other words, an embedding $I \to \overline{U}_i$ such that $\boundary I$ maps to $\{x,y\}$ and the preimage of $\boundary U_i$ equals $\boundary I$.
Then $c_1 \cup c_2$ is a Jordan curve $c$ that meets $M$ only in the points $x$ and $y$.   This Jordan curve divides the sphere into two components such that at least one component of $M \setminus \lbrace x, y \rbrace$ lies inside the circle and at least one lies outside.  Indeed $\partial U_i \subset M$ for each $i$ as $M$ is closed.  Thus $\{x,y\}$ is a cut pair of $M$, a contradiction.
\end{proof}

The following proposition was previously known to hold for the \Sierpinski\ carpet \cite{Whyburn_Sierpinski,KapovichKleiner00}.  We extend it to the setting of planar Peano continua.

\begin{prop} \label{prop:Planar} 
Let $M\subset S^2$ be a planar Peano continuum with no cut points and no cut pairs.
Then each homeomorphism of $M$ extends to a homeomorphism of $S^2$.
Furthermore any convergence group action of a group $G$ on $M$ extends to a convergence group action on $S^2$ with limit set contained in $M$.
\end{prop}

\begin{proof}
According to Theorem~\ref{thm:PlanarPeano}, the closure of each complementary region $U$ of $S^2\setminus M$ is homeomorphic to a closed disc whose boundary $\boundary U$ is a Jordan curve and each pair of these discs intersects in at most one point.
Since $M \subset S^2$, any embedded circle in $M$ that does not bound a complementary component of $M$ in $S^2$  must separate $M$.
We will show conversely that any circle $\boundary U$ that bounds a complementary region $U$ does not separate $M$.

Fix a complementary region $U$.
We will see that $M\setminus \boundary U$ is path connected.
The Jordan curve $\boundary U$ separates $S^2$ into $U$ and a disc $\hat{U}$ containing $M \setminus \boundary U$.
Choose a subset $F\subset S^2$ containing one point $q_{U'}$ for each component $U'$ of $S^2 \setminus M$ as follows.
Let $q_{U'}$ be the unique point in $\overline{U} \cap \overline{U}'$ if such a point exists, and otherwise let $q_{U'}$ be any point of $U'$.
Points of the first type lie in $\boundary \hat{U}$, while points of the second type lie in $\hat{U}$.
Since $F$ is countable, $\hat{U} \setminus F$ is path connected.

We now show that $\hat{U} \setminus F$ retracts onto $M \setminus \boundary U = M \cap \hat{U}$.
Note that $\hat{U} \setminus F$ is obtained from $M \cap \hat{U}$ by adding a countable number of punctured discs $\overline{U}'\setminus \{q_{U'}\}$ of two types.
If $\overline{U}'$ is a complementary component intersecting $\overline{U}$ in the point $q_{U'}$ then $q_{U'}$ lies on the boundary of $\overline{U}'$ and $\overline{U}'\setminus \{q_{U'}\}$ is homeomorphic to $\R \times [0,\infty)$.
But if $\overline{U}'$ is disjoint from $\overline{U}$ then $q_{U'}$ is in the interior of $U'$
and $\overline{U}'\setminus \{q_{U'}\}$ is homeomorphic to $S^1 \times [0,\infty)$.
In either case the punctured disc or boundary punctured disc $\overline{U}'\setminus \{q_{U'}\}$ retracts onto $\boundary U' \setminus \{q_{U'}\}$, a subset of $M \cap \hat{U}$.
The retraction $r \colon \hat{U} \setminus F \to M \cap \hat{U}$
is defined piecewise; on $M \cap \hat{U}$ it equals the identity function and on each punctured disc $\overline{U}'\setminus \{q_{U'}\}$ it is a retraction onto the boundary of the punctured disc.
By Theorem~\ref{thm:PlanarPeano}(\ref{item:Nullity}), this retraction is continuous.
Since $\hat{U} \setminus F$ retracts onto $M \setminus \boundary U$, the latter space is path connected.
In particular, we have shown that $\boundary U$ does not separate $M$.

Any homeomorphism $h$ of $M$ leaves its family of nonseparating circles invariant.  Therefore, it permutes the boundary circles of the complementary regions in $S^2$.
Any homeomorphism of the circle extends to a homeomorphism of the disc, establishing the first claim.
Extending the action of $G$ on $M$ to a convergence group action on $S^2$ requires a bit more care.
Choose a representative $U$ from each orbit of nonseparating circles.  For each such $U$, the stabilizer $H_U$ is a convergence group acting on $\partial U= S^1$. Then by \cite{Tukia_Fuchsian, Gabai92, CassonJungreis94} there exists a Fuchsian action of $H_U$ on the hyperbolic plane whose boundary action is topologically conjugate to the given action on $\boundary U$.
We extend the action of $H_U$ on $\boundary U$ to $\bar U$ by identifying $\bar U$ with $\bar\Hyp^2$. 
Note that the action of $H_U$ on $\bar U$ is a convergence group whose limit set is contained in $\boundary U$.
The action of $G$ on $M$ then extends equivariantly to all discs in the complement of $M$. 
By Theorem~\ref{thm:PlanarPeano}(\ref{item:Nullity}), each extension is continuous. The action on $S^2$ satisfies the convergence property of Definition~\ref{def:convergence}, since any collapsing sequence for the action on $M$ is also a collapsing sequence for the action on $S^2$.
By construction, the limit set of this action is contained in $M$.
\end{proof}

The following is an immediate consequence of Theorem~\ref{thm:nocutpairs} and Proposition~\ref{prop:Planar}. 

\begin{thm}
\label{thm:RigidExtends}
Suppose $G$ is one ended and $(G, \mathcal{P})$ is relatively hyperbolic. If $(G, \mathcal{P})$ is rigid 
and $M=\boundary(G,\mathcal{P})$ topologically embeds in $S^2$, the action of $G$ on $M$ extends to a convergence group action on $S^2$ with limit set $M$.  \qed
\end{thm}

The following result generalizes Theorem~\ref{thm:RigidExtends}, and follows immediately from Proposition~\ref{prop:RigidNoCutPairs}, Lemma~\ref{lem:MZ:Planar}, and Proposition~\ref{prop:Planar}.

\begin{thm}
\label{thm:Planargeneral}
Suppose $G$ is one ended, and $(G, \mathcal{P})$ is relatively hyperbolic with boundary $M=\boundary(G,\mathcal{P})$ a planar Peano continuum with no cut points.
Let $G_Z$ be a rigid piece of the JSJ decomposition over $2$--ended subgroups.  
Every homeomorphism of $M_Z = \partial(G_Z,\mathcal{Q}_Z)$ extends to a homeomorphism of $S^2$.
Furthermore the extension can be chosen so that the minimal convergence action of $G_Z$ on $M_Z$ extends to a convergence group action on $S^2$ with limit set equal to $M_Z$. \qed
\end{thm} 

\begin{proof}[Proof of Theorem~\ref{thm:VirtualSurfaceGroup}]
We claim that any peripheral subgroup is contained in a vertex stabilizer of the JSJ decomposition.  Indeed, the JSJ decomposition is relative to $\mathcal{P}$ (see Section~\ref{sec:prelim}) and any homeomorphism of $\partial(G, \mathcal{P})$ preserves this decomposition, see Theorem~\ref{thm:JSJ}.   Since the boundary does not contain a cut point, any peripheral subgroup is either a cusp group of a virtually Fuchsian subgroup (so two ended) or contained in a rigid vertex group that is not separated by any exact inseparable cut pair.   
As two-ended groups are virtual surface groups, it suffices to consider the rigid case.

By Theorem~\ref{thm:Planargeneral}, the parabolic action of the group $P$ on $M_Z$ extends to a parabolic action on $S^2$. Therefore $P$ also acts properly on the plane $S^2 \setminus \{a\}$.
Every peripheral subgroup $P$ of a finitely generated relatively hyperbolic group is itself finitely generated by Osin \cite[Prop.~2.29]{Osin06}.
Therefore $P$ is virtually a surface group, by Corollary~\ref{cor:VirtuallyTorsionFree}.  
\end{proof}


\bibliographystyle{alpha}
\bibliography{chruska.bib}

\def\polhk#1{\setbox0=\hbox{#1}{\ooalign{\hidewidth
  \lower1ex\hbox{$\,\lhook$}\hidewidth\crcr\unhbox0}}}
  \def\RomanianComma#1{\setbox0=\hbox{#1}{\ooalign{\hidewidth
  \lower1.2ex\hbox{$\mspace{1mu}^{,}$}\hidewidth\crcr\unhbox0}}}
\begin{thebibliography}{Bow99b}

\bibitem[AS60]{AhlforsSario_RiemannSurfaces}
Lars~V. Ahlfors and Leo Sario.
\newblock {\em Riemann surfaces}.
\newblock Princeton Mathematical Series, No. 26. Princeton University Press,
  Princeton, N.J., 1960.

\bibitem[Ayr29]{Ayres29}
W.L. Ayres.
\newblock Continuous curves homeomorphic with the boundary of a plane domain.
\newblock {\em Fund.\ Math.}, 14:92--95, 1929.

\bibitem[Bea83]{Beardon83_DiscreteGroups}
Alan~F. Beardon.
\newblock {\em The geometry of discrete groups}, volume~91 of {\em Graduate
  Texts in Mathematics}.
\newblock Springer-Verlag, New York, 1983.

\bibitem[BKK02]{BestvinaKapovichKleiner02}
Mladen Bestvina, Michael Kapovich, and Bruce Kleiner.
\newblock Van {K}ampen's embedding obstruction for discrete groups.
\newblock {\em Invent. Math.}, 150(2):219--235, 2002.

\bibitem[Bow98]{Bowditch98JSJ}
Brian~H. Bowditch.
\newblock Cut points and canonical splittings of hyperbolic groups.
\newblock {\em Acta Math.}, 180(2):145--186, 1998.

\bibitem[Bow99a]{Bowditch99Connectedness}
B.H. Bowditch.
\newblock Connectedness properties of limit sets.
\newblock {\em Trans.\ Amer.\ Math.\ Soc.}, 351(9):3673--3686, 1999.

\bibitem[Bow99b]{Bowditch99Treelike}
B.H. Bowditch.
\newblock Treelike structures arising from continua and convergence groups.
\newblock {\em Mem.\ Amer.\ Math.\ Soc.}, 139(662):1--86, 1999.

\bibitem[Bow01]{Bowditch_Peripheral}
B.H. Bowditch.
\newblock Peripheral splittings of groups.
\newblock {\em Trans. Amer. Math. Soc.}, 353(10):4057--4082, 2001.

\bibitem[Bow12]{BowditchRelHyp}
B.H. Bowditch.
\newblock Relatively hyperbolic groups.
\newblock {\em Internat. J. Algebra Comput.}, 22(3):1250016, 66 pages, 2012.

\bibitem[BW13]{BigdelyWise_Combination}
Hadi Bigdely and Daniel~T. Wise.
\newblock Quasiconvexity and relatively hyperbolic groups that split.
\newblock {\em Michigan Math. J.}, 62(2):387--406, 2013.

\bibitem[BZ]{Benzvi_PathCon}
Michael Ben-Zvi.
\newblock Boundaries of groups with isolated flats are path connected.
\newblock arXiv:1909.12360.

\bibitem[Can91]{Cannon91}
J.W. Cannon.
\newblock The theory of negatively curved spaces and groups.
\newblock In T.~Bedford, M.~Keane, and C.~Series., editors, {\em Ergodic
  theory, symbolic dynamics, and hyperbolic spaces \textup{(}{T}rieste,
  1989\textup{)}}, Oxford Sci. Publ., pages 315--369. Oxford Univ. Press, New
  York, 1991.

\bibitem[CJ94]{CassonJungreis94}
Andrew Casson and Douglas Jungreis.
\newblock Convergence groups and {S}eifert fibered {$3$}--manifolds.
\newblock {\em Invent. Math.}, 118(3):441--456, 1994.

\bibitem[Dah03]{Dahmani03Combination}
F.~Dahmani.
\newblock Combination of convergence groups.
\newblock {\em Geom. Topol.}, 7:933--963, 2003.

\bibitem[Dah05]{Dahmaniparabolic}
Fran\c{c}ois Dahmani.
\newblock Parabolic groups acting on one-dimensional compact spaces.
\newblock {\em Internat. J. Algebra Comput.}, 15(5-6):893--906, 2005.

\bibitem[Das20]{Dasgupta_Thesis}
Ashani Dasgupta.
\newblock {\em Local connectedness of {B}owditch boundary of relatively
  hyperbolic groups}.
\newblock PhD thesis, University of Wisconsin--Milwaukee, 2020.

\bibitem[DH]{DasguptaHruska_LC}
A.~Dasgupta and G.C. Hruska.
\newblock Local connectedness of boundaries for relatively hyperbolic groups.
\newblock Preprint. arXiv:2204.02463.

\bibitem[Dre69]{Dress_NewmanThm}
Andreas Dress.
\newblock Newman's theorems on transformation groups.
\newblock {\em Topology}, 8:203--207, 1969.

\bibitem[DS05]{DrutuSapirTreeGraded}
C.~Dru{\cb{t}}u and M.~Sapir.
\newblock Tree-graded spaces and asymptotic cones of groups.
\newblock {\em Topology}, 44(5):959--1058, 2005.
\newblock With an appendix by D. Osin and M. Sapir.

\bibitem[Fre95]{Freden95}
Eric~M. Freden.
\newblock Negatively curved groups have the convergence property. {I}.
\newblock {\em Ann.\ Acad.\ Sci.\ Fenn. Ser.\ A I Math.}, 20(2):333--348, 1995.

\bibitem[Gab92]{Gabai92}
David Gabai.
\newblock Convergence groups are {F}uchsian groups.
\newblock {\em Ann. of Math. \textup{(}2\textup{)}}, 136(3):447--510, 1992.

\bibitem[GL11]{GuirardelLevitt_Canonical}
Vincent Guirardel and Gilbert Levitt.
\newblock Trees of cylinders and canonical splittings.
\newblock {\em Geom. Topol.}, 15(2):977--1012, 2011.

\bibitem[GL15]{GuirardelLevitt15_AutRelHyp}
Vincent Guirardel and Gilbert Levitt.
\newblock Splittings and automorphisms of relatively hyperbolic groups.
\newblock {\em Groups Geom. Dyn.}, 9(2):599--663, 2015.

\bibitem[GL17]{GuirardelLevitt}
Vincent Guirardel and Gilbert Levitt.
\newblock {\em J{SJ} decompositions of groups}, volume 395 of {\em
  Ast\'erisque}.
\newblock Soci\'et\'e Math\'ematique de France, Paris, 2017.

\bibitem[GM87]{GehringMartin87}
F.W. Gehring and G.J. Martin.
\newblock Discrete quasiconformal groups. {I}.
\newblock {\em Proc.\ London Math.\ Soc. \textup{(}3\textup{)}},
  55(2):331--358, 1987.

\bibitem[GM08]{GrovesManning08DehnFilling}
Daniel Groves and Jason~Fox Manning.
\newblock Dehn filling in relatively hyperbolic groups.
\newblock {\em Israel J. Math.}, 168:317--429, 2008.

\bibitem[GMS19]{GrovesManningSisto}
Daniel Groves, Jason~Fox Manning, and Alessandro Sisto.
\newblock Boundaries of {D}ehn fillings.
\newblock {\em Geom. Topol.}, 23(6):2929--3002, 2019.

\bibitem[GP01]{GaboriauPaulin01}
Damien Gaboriau and Fr{\'e}d{\'e}ric Paulin.
\newblock Sur les immeubles hyperboliques.
\newblock {\em Geom. Dedicata}, 88(1--3):153--197, 2001.

\bibitem[GP15]{GerasimovPotyagailo15_NonFG}
V.~Gerasimov and L.~Potyagailo.
\newblock Non-finitely generated relatively hyperbolic groups and {F}loyd
  quasiconvexity.
\newblock {\em Groups Geom. Dyn.}, 9(2):369--434, 2015.

\bibitem[GP16]{GerasimovPotyagailo16_Similar}
V.~Gerasimov and L.~Potyagailo.
\newblock Similar relatively hyperbolic actions of a group.
\newblock {\em Int. Math. Res. Not. IMRN}, 2016(7):2068--2103, 2016.

\bibitem[Gur05]{Guralnik05}
Dan~P. Guralnik.
\newblock Ends of cusp-uniform groups of locally connected continua. {I}.
\newblock {\em Internat. J. Algebra Comput.}, 15(4):765--798, 2005.

\bibitem[Ha{\"{\i}}15]{Haissinsky_Invent}
Peter Ha{\"{\i}}ssinsky.
\newblock Hyperbolic groups with planar boundaries.
\newblock {\em Invent. Math.}, 201(1):239--307, 2015.

\bibitem[Hau19]{HaulmarkRelHyp}
Matthew Haulmark.
\newblock Local cut points and splittings of relatively hyperbolic groups.
\newblock {\em Algebr. Geom. Topol.}, 19(6):2795--2836, 2019.

\bibitem[HH]{HaulmarkHruska_JSJ}
Matthew Haulmark and G.~Christopher Hruska.
\newblock On canonical splittings of relatively hyperbolic groups.
\newblock To appear in \emph{Israel J.\ Math.}. arXiv:1912.00886.

\bibitem[HL]{petercyril}
Peter Ha{\"{\i}}ssinsky and Cyril Lecuire.
\newblock Quasi-isometric rigidity of $3$--manifold groups.
\newblock arXiv:2005.06813.

\bibitem[HPW]{HPWpreprint}
Peter Ha\"{i}ssinsky, Luisa Paoluzzi, and Genevieve Walsh.
\newblock Relatively hyperbolic groups with {S}chottky set boundary.
\newblock In preparation.

\bibitem[Hru10]{Hruska10RelQC}
G.C. Hruska.
\newblock Relative hyperbolicity and relative quasiconvexity for countable
  groups.
\newblock {\em Algebr. Geom. Topol.}, 10(3):1807--1856, 2010.

\bibitem[HST20]{HruskaStarkTran_DontAct}
G.~Christopher Hruska, Emily Stark, and Hung~Cong Tran.
\newblock Surface group amalgams that (don't) act on $3$--manifolds.
\newblock {\em Amer. J. Math.}, 142(3):885--921, 2020.

\bibitem[KK00]{KapovichKleiner00}
Michael Kapovich and Bruce Kleiner.
\newblock Hyperbolic groups with low-dimensional boundary.
\newblock {\em Ann. Sci. \'Ecole Norm. Sup. \textup{(}4\textup{)}},
  33(5):647--669, 2000.

\bibitem[Kol06]{Kolev06}
Boris Kolev.
\newblock Sous-groupes compacts d'hom\'{e}omorphismes de la sph\`ere.
\newblock {\em Enseign. Math. \textup{(}2\textup{)}}, 52(3-4):193--214, 2006.

\bibitem[Kur68]{Kuratowski_VolII}
K.~Kuratowski.
\newblock {\em Topology. {V}ol. {II}}.
\newblock Academic Press, New York-London; Pa\'{n}stwowe Wydawnictwo Naukowe
  Polish Scientific Publishers, Warsaw, 1968.

\bibitem[Mil06]{Milnor_Dynamics}
John Milnor.
\newblock {\em Dynamics in one complex variable}, volume 160 of {\em Annals of
  Mathematics Studies}.
\newblock Princeton University Press, Princeton, NJ, third edition, 2006.

\bibitem[MOY19]{Matsudarel}
Y.~Matsuda, S.~Oguni, and S.~Yamagata.
\newblock On relative hyperbolicity for a group and relative quasiconvexity for
  a subgroup.
\newblock {\em Tokyo J. Math.}, 42(1), 2019.

\bibitem[MR99]{MihalikRuane99}
Michael Mihalik and Kim Ruane.
\newblock {$\CAT(0)$} groups with non-locally connected boundary.
\newblock {\em J. London Math. Soc. \textup{(}2\textup{)}}, 60(3):757--770,
  1999.

\bibitem[MS89]{MartinSkora89}
Gaven~J. Martin and Richard~K. Skora.
\newblock Group actions of the {$2$}--sphere.
\newblock {\em Amer. J. Math.}, 111(3):387--402, 1989.

\bibitem[Osi06a]{Osin06_Elementary}
D.V. Osin.
\newblock Elementary subgroups of relatively hyperbolic groups and bounded
  generation.
\newblock {\em Internat. J. Algebra Comput.}, 16(1):99--118, 2006.

\bibitem[Osi06b]{Osin06}
D.V. Osin.
\newblock Relatively hyperbolic groups: {I}ntrinsic geometry, algebraic
  properties, and algorithmic problems.
\newblock {\em Mem. Amer. Math. Soc.}, 179(843):1--100, 2006.

\bibitem[PS06]{PapasogluSwenson06}
Panos Papasoglu and Eric Swenson.
\newblock From continua to {$\mathbb{R}$}--trees.
\newblock {\em Algebr. Geom. Topol.}, 6:1759--1784, 2006.

\bibitem[PS11]{PapasogluSwenson11_Cactus}
Panos Papasoglu and Eric Swenson.
\newblock The cactus tree of a metric space.
\newblock {\em Algebr. Geom. Topol.}, 11(5):2547--2578, 2011.

\bibitem[Ser77]{Serre77}
Jean-Pierre Serre.
\newblock {\em Arbres, amalgames, {${\rm SL}\sb{2}$}}, volume~46 of {\em
  Ast\'erisque}.
\newblock Soci\'et\'e Math\'ematique de France, Paris, 1977.
\newblock Written in collaboration with Hyman Bass.

\bibitem[Tra13]{Tran13}
Hung~Cong Tran.
\newblock Relations between various boundaries of relatively hyperbolic groups.
\newblock {\em Internat. J. Algebra Comput.}, 23(7):1551--1572, 2013.

\bibitem[Tuk88]{Tukia_Fuchsian}
Pekka Tukia.
\newblock Homeomorphic conjugates of {F}uchsian groups.
\newblock {\em J. Reine Angew. Math.}, 391:1--54, 1988.

\bibitem[Tuk94]{Tukia94}
Pekka Tukia.
\newblock Convergence groups and {G}romov's metric hyperbolic spaces.
\newblock {\em New Zealand J. Math.}, 23(2):157--187, 1994.

\bibitem[TW]{TshWa}
Bena Tshishiku and Genevieve Walsh.
\newblock On groups with {$S^2$} {B}owditch boundary.
\newblock To appear in \emph{Groups Geom. Dyn.} arXiv:1710.09018.

\bibitem[Why58]{Whyburn_Sierpinski}
G.T. Whyburn.
\newblock Topological characterization of the {S}ierpi\'{n}ski curve.
\newblock {\em Fund. Math.}, 45:320--324, 1958.

\bibitem[Wil49]{Wilder_Topology}
Raymond~Louis Wilder.
\newblock {\em Topology of {M}anifolds}.
\newblock American Mathematical Society Colloquium Publications, Vol. 32.
  American Mathematical Society, New York, N. Y., 1949.

\bibitem[Yam04]{Yaman04}
Asl{\i} Yaman.
\newblock A topological characterisation of relatively hyperbolic groups.
\newblock {\em J. Reine Angew. Math.}, 566:41--89, 2004.

\bibitem[Yan14]{Yang14_Peripheral}
Wen-yuan Yang.
\newblock Peripheral structures of relatively hyperbolic groups.
\newblock {\em J. Reine Angew. Math.}, 689:101--135, 2014.

\end{thebibliography}

\end{document}